\documentclass[twoside, 12pt]{article}
\usepackage{a4wide, amsthm, amssymb, amscd, mathptmx}

\usepackage{amsfonts}

\usepackage[all]{xy}\CompileMatrices\SelectTips{cm}{12}

\theoremstyle{plain}
\newtheorem{Thm}{\sc Theorem}[section]
\newtheorem{Theorem}[Thm]{\sc Theorem}
\newtheorem*{Theorem*}{\sc Theorem}
\newtheorem{Corollary}[Thm]{\sc Corollary}
\newtheorem*{Lemma*}{\sc Lemma}

\newtheorem*{Proposition*}{\sc Proposition}
\newtheorem{Lemma}[Thm]{\sc Lemma}

\theoremstyle{definition}

\theoremstyle{remark}

\newtheorem{Example}[Thm]{Example}
\newtheorem*{Example*}{Example}
\newtheorem*{Remark*}{Remark}

\newcommand{\K}{{\mathcal K}}
\newcommand{\cL}{{\mathcal L}}

\renewcommand{\O}{{\mathcal O}}

\newcommand{\X}{{\mathcal X}}

\newcommand{\PP}{\mathbb{P}}

\newcommand{\ZZ}{\mathbb{Z}}

\newcommand{\can}{\mathop{\rm can}}

\newcommand{\Pic}{{\mathop{\rm Pic \,}}}

\newcommand{\rk}{\mathop{\rm rk}}

\newcommand{\Spec}{\mathop{\rm Spec}}

\begin{document}

\pagestyle{myheadings} \markboth{\rm A.\ Langer}{\rm Restriction
theorem for strong stability}

\title{{\large A note on restriction theorems for semistable
sheaves}}
\author{Adrian Langer}
\date{\today}

\maketitle

{\sc Addresses:} 1. Institute of Mathematics, Warsaw University,
ul.\ Banacha 2,
02-097 Warszawa, Poland,\\

2. Institute of Mathematics, Polish Academy of Sciences, ul.
\'Sniadeckich 8, 00--956 Warszawa, Poland, e-mail: {\tt
alan@mimuw.edu.pl}

\section*{Introduction}

The main result of this note is a generalization of the following theorem from
the characteristic zero case to positive characteristic.

\begin{Theorem} \label{rest-0}
Let $X$ be a smooth projective  variety  of dimension $n$ defined
over an algebraically closed field $k$ of characteristic $0$. Let
$H$ be a very ample divisor on $X$ and let $E$ be an
$H$-semistable torsion free sheaf of rank $r\ge 2$ on $X$. Let us
take an integer $a$ such that
$${{a+n \choose n}} > \frac{1}{2}
\left( {\max\{\frac{r^2-1}{4},1 \}}H^n +1\right) \Delta (E)H^{n-2}
+1 .$$ Then for a general divisor $D\in |aH|$ the restriction
$E_D$ is $H_D$-semistable.
\end{Theorem}

To avoid introducing a complicated notation we defer a precise
formulation of the generalization of Theorem \ref{rest-0} to
Section 2 (see Theorem \ref{rest-ss}). The above stated theorem is
also new although a slightly weaker result can be obtained by
combining Flenner's restriction theorem and the author's version
of Bogomolov's restriction theorem (see \cite[Theorem 5.2]{La1}).
In positive characteristic the analogue of Theorem \ref{rest-0}
(i.e., Theorem \ref{rest-ss}) is new even in the well-known and
very much studied case of rank $2$ vector bundles on $\PP ^2$.

Obviously, the theorem in positive characteristic implies the
characteristic zero version by usual arguments involving reduction
modulo large characteristic. Hence in the following we consider
only the positive characteristic version. In positive
characteristic $p$, the theorem improves earlier bounds on degree
of hypersurfaces for which the restriction is semistable from
roughly the discriminant to the $n$-th root of discriminant. In
particular, it shows that for every strongly semistable sheaf $E$
(of rank $r<p+2$) on a variety of dimension at least $2$ there
exists a computable integer $a_0$ such that for all integers
$a>a_0$ the restriction of $(F^m)^*E$ is semistable on the generic
hypersurface of degree $ap^m$. This generalizes the main result of
\cite{MT} (see \cite[Theorem 2.7]{MT}), where a similar theorem
was proven for some special classes of varieties. In fact, for
varieties considered in \cite{MT} (and many more) we prove that
there exists a restriction theorem for strong semistability (see
Theorem \ref{rest-sss}).

The proof of our theorem is a combination of proofs of the
Grauert--M\"ulich, Flenner's and Bogomolov's restriction theorems.
The main new features are a non-separable descent and the use of
author's version of Bogomolov's inequality in positive
characteristic (see Theorem \ref{Bogomolov}) to elementary
transformations along non-reduced divisors.

\medskip

The paper is organized as follows. In Section 1 we recall a few
basic results. In Section 2 we prove the main result: a new
restriction theorem for semistability. In Section 3 we prove a
restriction theorem for strong semistability on varieties with
certain non-negativity constrains on the cotangent bundle (this
class includes, e.g., non-supersingular K3 surfaces and most of
Fano and Calabi--Yau varieties).

\medskip

\subsection{Notation and conventions}

Let $S$ be a scheme defined over an algebraically closed field $k$
of characteristic $p$. Then by $F_S^r :S\to S$ we denote the
$r$-th \emph{absolute Frobenius morphism} of $S$ (i.e., the
$p^r$-th power mapping on $\O _S$).

Let $X$ be an $S$-scheme and let $X^{(r)}_S$ denote the fiber
product of $X$ and $S$ over the $r$-th Frobenius morphism $F^r_S:
S\to S $.  The $r$-th absolute Frobenius morphism of $X$ induces
the $r$-th \emph{relative Frobenius morphism} $F^r_{X/S}: X\to
X^{(r)}_S$. This morphism decomposes into a composition of $r$
relative Frobenius morphisms $F^r_{X/S}=F_{X^{(r-1)}_S/S} \dots
 F_{X^{(1)}_S/S}  F_{X/S}.$ For a morphism $\pi: X\to S$
the induced morphism $X^{(r)}_S\to S$ is denoted by $\pi
^{(r)}_S$. The above notation is best illustrated by the following
diagram:
$$ \xymatrix{ &X\ar[rd]_{\pi}\ar@(ur,ul)[rr]^{\mbox{\scriptsize $F_{X}^r$}}\ar[r]^{\mbox{\tiny $F_{X/S}^r$}}&
{X_S^{(r)}}\ar[d]^{\mbox{\scriptsize$\pi^{(r)}_S$}} \ar[r] & X\ar[d]^{\pi}\\
&&S\ar[r]^{\mbox{\scriptsize $F_S^r$}}&S\\}$$ When $S=\Spec k$
then to simplify notation we denote $X^{(r)}_S$ by $X^{(r)}$. As
we will need a few different Frobenius morphisms we will use a
slightly inconsistent notation and for a morphism $\pi: X\to S$,
$\pi ^{(r)}$ will denote not the morphism  $\pi ^{(r)}_k$ but the
morphism $X^{(r)}\to S^{(r)}$ induced by $\pi$.

Semistability in this paper will always mean slope semistability with
respect to the considered ample divisor.

\section{Preliminaries}

In this section we gather a few auxiliary results.

\medskip

Let $(X, H)$ be a smooth polarized variety of dimension $n$
defined over an algebraically closed field $k$ of characteristic
$p$. Let $E$ be a rank $r$ torsion free sheaf on $X$. Then one can
define the \emph{slope} of $E$ by $\mu (E)=c_1E\cdot H ^{n-1}/r$.
 The slope of a maximal destabilizing
subsheaf of $E$ is denoted by $\mu _{\max}(E)$ and that of minimal
destabilizing quotient by $\mu _{\min} (E)$. The
\emph{discriminant} of $E$ is defined by $\Delta
(E)=2rc_2(E)-(r-1)c_1^2(E)$.

Let us recall that in positive characteristic the degree of
instability of Frobenius pull backs of $E$ is measured by the
following (well-defined and finite) number:
$$L_{\max} (E) =\lim _{k\to \infty} \frac{\mu _{\max} ((F^k)^*E)}{p^k}.$$
 Let us set
$$L_X= \left\{\begin{array} {ll} \frac{L_{\max}(\Omega_X)}{p}&\quad
\hbox{if $\mu_{\max}(\Omega _X)>0$,}\\ 0 &\quad
\hbox{if $\mu_{\max}(\Omega _X)\le 0$.}\\
\end{array}
\right.$$

The following theorem is a slight strengthening of \cite[Theorem
5.1]{La1}. It can be proved in the same way except that at the end
of proof one needs to use \cite[Theorem 2.17]{La2} instead of
\cite[Theorem 3.3]{La1}.

\begin{Theorem} \label{Bogomolov}
For an arbitrary torsion free sheaf $E$ we have the following
inequality:
$$H^n\cdot
\Delta (E)H^{n-2}+ +r^2(\mu_{\max} (E)-\mu (E)) (\mu (E)
-\mu_{\min} (E) ) +
 \frac{r^2(r-1)^2}{4} {L_X}^2 \ge 0.$$
\end{Theorem}

\medskip
We will also need a corollary to the following theorem of
Ilangovan, Mehta and Pa\-ra\-mes\-wa\-ran:

\begin{Theorem} (see \cite[Theorem 3.1]{IMP})
Let $G$ be a connected reductive algebraic group over  $k$ and let
$G\to SL(V)$ be a low height representation (see \cite[2.1]{IMP}
for the definition of low height representations). Then  for every
semistable principal $G$-bundle $E$ on $X$, the induced vector
bundle $E(V)$ is also semistable with respect to the same
polarization.
\end{Theorem}

\begin{Corollary} \label{IMP}
If $E_1$ and $E_2$ are semistable sheaves and $p>\rk E_1+\rk E_2
-2$ then $E_1\otimes E_2$ is semistable.
\end{Corollary}

Finally, let us recall the following result which in
characteristic zero was proven by H. Flenner (see \cite[Corollary
2.2]{Fl}) and in arbitrary characteristic it was (almost) proven
by H. Brenner (see \cite[Corollary 7.1]{Br}):

\begin{Theorem} \label{Brenner}
The kernel $K_a$ of the evaluation map $H^0(\O_{\PP ^n}(a))\otimes
\O_{\PP^n}\to \O_{\PP ^n}(a)$ is semi\-stable.
\end{Theorem}

In fact, Brenner in his paper uses the characteristic zero
assumption when using Klyachko's results (see \cite[remark on
p.~416, before Lemma 6.1]{Br}). However, it was remarked by  M.
Perling in \cite{Pe} that Klyachko's results are valid in an
arbitrary characteristic. In fact, Perling's description is better
suited to generalizing \cite[Lemma 6.1]{Br} as it uses toric
sheaves rather than just toric vector bundles. A weaker form of
the above theorem, but still sufficient for our applications, is
the main result of Trivedi's paper \cite{Tr}. In particular,
Trivedi proves that
$$\mu _{\max}(K_a^*)\le \frac {a}
{{{ \lceil{\frac{a}{2}}\rceil +n-1}\choose {n-1}}} .$$

\section{New restriction theorem for semistability}

\begin{Theorem} \label{rest-ss}
Let $k$ be an algebraically closed field of characteristic $p>0$.
Let $X$ be a smooth $n$-dimensional ($n\ge 2$) $k$-variety with a
very ample divisor $H$. Let $E$ be an $H$-semistable torsion free
sheaf of rank $r\ge 2$ on $X$. Let us assume that $p\ge r-1$ and
let us take an integer $a$ such that
$${{a+n \choose n} } > \frac{1}{2}\left({\max\{\frac{r^2-1}{4},1 \}}H^n +1\right)
\left(  \Delta (E)H^{n-2}+\frac{r^2(r-1)^2{L_X}^2}{4H^n} \right)
+1.$$ Then for a general divisor $D\in |aH|$ the restriction $E_D$
is $H_D$-semistable.
\end{Theorem}

\begin{proof}
Let us assume that the restriction of $E$ to a general divisor in
$|aH|$ is not $H$-semistable.

Let $\Pi$ denote the complete linear system $|aH|$. Let $Z=\{
(D,x)\in \Pi \times X: x\in D\}$ be the incidence variety with
projections $p: Z\to \Pi$ and $q: Z\to X$. Let $Z_t$ denote the
scheme theoretic fibre of $p$ over $t\in \Pi$.  Let $0\subset
E_0\subset E_1\subset \dots \subset E_l=q^*E$ be the relative
Harder--Narasimhan filtration of $q^*E$ with respect to $p$ (see
\cite[Theorem 2.3.2 and Remark 2.3.3]{HL}). By definition there
exists a nonempty open subset $U$ of $\Pi$ such that all quotients
$E^i=E_i/E_{i-1}$ are flat over $U$ and such that for every $t\in
U$ the fibres $(E_{\bullet})_t$ form the Harder--Narasimhan
filtration of $E_t=(q^*E)_{Z_t}$ (we can also assume that all
fibers $Z_t$  for $t\in U$ are smooth).

Now the proof splits into several steps.

\medskip

\emph{Step 1.} In this step we find an upper bound on a certain
difference of slopes of consecutive quotients in the
Harder--Narasimhan filtration of restriction of $E$ to a general
divisor in $|aH|$ (see inequality (\ref{ineq})). This is done
using an inseparable descent.

Assume that the filtration $E_{\bullet}$ descends under the $s$-th
relative Frobenius morphism $F^{s}_{Z/X}: Z\to Z^{s}_X$ but it
does not descend under the $(s+1)$-th relative Frobenius morphism
$F^{s+1}_{Z/X}: Z\to Z^{(s+1)}_X$.   To simplify notation let us
denote $Z^{(s)}_X$ by $Z'$ and $q^{(s)}_X$ by $q'$. Let
$E_{\bullet}'$ be a filtration of $(q')^* E$ such that
$E_{\bullet}\simeq (F_{Z/X}^{(s)})^*E_{\bullet}'$. Since
$$(q')^* E=(F_{Z'/X})^* ((q_X^{(s+1)})^*E)$$
there exists a canonical relative connection
$$\nabla _{\can} : (q')^*E\to \Omega _{Z'/X} \otimes (q')^* E .$$
By our assumption and Cartier's theorem (see, e.g., \cite[Theorem
5.1]{Ka}) there exists some index $i_0$ such that $E_{i_0}'$ is
not preserved by $\nabla_{\can}$. Then the connection induces a
non-zero $\O_{Z'}$-homomorphism
$$E_{i_0}'\to \Omega _{Z'/X}\otimes  ((q')^*E)/E_{i_0}'.$$
After pulling-back to $Z$ we get a non-zero $\O_Z$-homomorphism
$$E_{i_0}\to (F^{s}_{Z/X})^*\Omega _{Z'/X}\otimes  (q^* E)/E_{i_0}.$$

\bigskip
After restricting to a general fibre $Z_t$ of $p$ we see that
$$\mu_{\min}((E_{i_0}\otimes (q^* E/E_{i_0})^*)_{Z_t})\le
\mu_{\max}(((F^{s}_{Z/X})^*\Omega _{Z'/X})_{Z_t}).$$

By assumption, quotients $E^i_{Z_t}$ are semistable for a general
point $t\in \Pi$. As the sum of ranks of $E^i$ and $E^j$ is less
than or equal to $r\le p+1$, the Ilangovan--Mehta--Parameswaran
theorem (see Corollary \ref{IMP}) implies semistability of sheaves
$E^i_{Z_t}\otimes (E^j_{Z_t})^*$  for all $i,j$. Let us set
$r_i=\rk (E^i_{Z_t})$ and $\mu_i=\mu (E^i_{Z_t})$. Then by the
above we have
\begin{equation}\label{ineq}
\mu_{i_0}-\mu_{{i_0}+1}\le
\mu_{\max}(((F^{s}_{Z/X})^*\Omega_{Z'/X})_{Z_t}).
\end{equation}

\medskip
\emph{Step 2.} In this step we introduce some notation and we
recall a lower bound on differences $\mu_{i}-\mu _{i+1}$ (see
inequality (\ref{mu-diff})).

Let $\K$ be the kernel of the evaluation map $H^0(X, \O
_X(aH))\otimes \O_X\to \O _X (aH)$. Then $Z=\PP _X(\K ^{\vee})$
and $q$ corresponds to the natural projection $\PP _X(\K
^{\vee})\to X$. It is easy to see that $Z'=\PP _X((F^s)^*\K
^{\vee})$ and
$(F^s_{Z/X})^*\O_{\PP_X((F^s)^*\K^{\vee})}(1)=\O_{\PP
_X(\K^{\vee})}(p^s)$. Let us note that the Picard group of $\PP _X
((F^s)^*\K ^{\vee})$ is generated over $(q^{(s)})^*\Pic X$ by
$\O_{\PP((F^s)^*\K ^{\vee})}(1)$. Therefore we can write $\det
(E^i)'\simeq \O_{Z'}(-b_i')\otimes (q^{(s)})^*\cL_i$ for some
integers $b_i'$ and line bundles $\cL_i$ on $X$. Then $\det
E^i\simeq \O_{Z}(-p^sb_i')\otimes q^*\cL_i$. We know that
$\mu_i>\mu_{i+1}$ and $\mu_i=\frac{\deg E^i_{Z_t}}{r_i}$. By the
above $\deg E^i_{Z_t}=\deg (\cL_i|_{Z_s})=a L_i H^{n-1},$ where
$L_i=c_1\cL_i$. Hence
\begin{equation}\label{mu-diff}
\mu_i-\mu_{i+1}= \frac{a L_i H^{n-1}}{r_i}-\frac{ aL_{i+1}
H^{n-1}}{r_{i+1}}\ge \frac{a}{\max\{\frac{r^2-1}{4},1 \}}.
\end{equation}

\medskip

\emph{Step 3.} In this step we find a precise bound on
$\mu_{\max}(((F^{s}_{Z/X})^*\Omega_{Z'/X})_{Z_t})$. This step can
be summed up in the following lemma, whose proof is similar to
proof of Flenner's restriction theorem (see \cite[Proposition
1.10]{Fl}; we will use \cite[Theorem 7.1.1]{HL} as reference as it
will be easier to follow the proof):

\begin{Lemma}
If $ap^s <{a+n \choose n}-1$ then
$$\mu_{\max}(((F^{s}_{Z/X})^*\Omega _{Z'/X})_{Z_t})\le
\frac{a^2p^sH^n}{{a+n \choose n} -ap^s-1}.$$
\end{Lemma}

\begin{proof}
First we need to reduce to the case $X=\PP ^n$. This can be done
in the same way as Steps 2 and 3 in proof of \cite[Theorem
7.1.1]{HL} and therefore we skip it. Let us set $N={a+n \choose
n}$ and let $K$ be the kernel of the evaluation map $H^0(\O_{\PP
^n}(a))\otimes \O_{\PP^n}\to \O_{\PP ^n}(a)$. As in Step 4 of
proof of \cite[Theorem 7.1.1]{HL}, using the Veronese embedding
$\PP^n \hookrightarrow \PP^{N-1}$ one can see that we only need to
show that
$$\mu _{\max}(((F^s)^*K^*)_D)\le \frac{a^2p^s}{N -ap^s-1},$$
where $D$ is a general degree $a$ hypersurface in $\PP^n$. Now let
us recall that $K$ is semistable (see Theorem \ref{Brenner}).
Since every semistable bundle on $\PP ^n$ is strongly semistable
(as $\mu _{\max} (\Omega_ {\PP ^n})<0$) we see that $(F^s)^*K^*$
is also semistable. Moreover, by the Ramanan--Ramanathan theorem
(see \cite[Theorem 3.23]{RR}) the bundles $\bigwedge ^q
((F^s)^*K^*)$ are also semistable for all $q$. Hence
$$H^0(\PP^n ,{\bigwedge} ^q ((F^s)^*K^*) \otimes \O_{\PP ^n}(b))=0$$
if $0>\mu (\bigwedge ^q ((F^s)^*K^*) \otimes \O_{\PP
^n}(b))=b-qp^s\mu (K)$, which is equivalent to
$b<-\frac{qap^s}{N-1}$. Now we can come back to Step 4 of proof of
\cite[Theorem 7.1.1]{HL} and finish in the same way.
\end{proof}

\medskip
\emph{Step 4.} Gathering information obtained in previous steps we
find in this step a lower bound on $ap^s$ (see inequality
(\ref{Flenner})), where $s$ is the same as in Step 1.

Lemma from Step 3, together with inequalities (\ref{ineq}) and
(\ref{mu-diff}) imply that if $ap^s <{a+n \choose n}-1$ then
$$\frac{a}{\max\{\frac{r^2-1}{4},1 \}}\le \frac{a^2p^sH^n}{{a+n \choose n} -ap^s-1}.$$
This inequality is equivalent to
\begin{equation} \label{Flenner}
ap^s\ge \frac{1}{H^n\max\{\frac{r^2-1}{4},1 \} +1} \left( {a+n
\choose n}-1\right).
\end{equation}
Obviously, the above inequality is also satisfied  if $ap^s \ge
{a+n \choose n}-1$.

\medskip

\emph{Step 5.} This step containing  the rest of the proof is
quite lengthy. It is devoted to finding an upper bound on $ap^s$.
First we prove that the restrictions of $E$ to some non-reduced
divisors are unstable. Then the strategy is quite similar to the
author's proof of Bogomolov type restriction theorem \cite[Theorem
5.2]{La1}. Namely, we can perform elementary transformations with
respect to these non-reduced divisors and use Theorem
\ref{Bogomolov} to find an upper bound on degree of divisors for
which the restriction is unstable. One of the new features is that
to get interesting inequalities, instability of restrictions to
non-reduced divisors needs to be understood for locally free
subsheaves rather than for all pure subsheaves.

Let us consider the commutative diagram
$$
\xymatrix{ Z \ar@(ur,ul)[rr]^{F^s_{Z/k}}\ar[r]_{F^s_{Z/X}}
\ar[d]_{p}&Z'\ar[r] &
Z^{(s)} \ar[d]^{p^{(s)}} \\
\Pi \ar[rr]_{F^{s}_{\Pi/k}} && \Pi^{(s)} }
$$
Let $p'$ be the composition of $Z'\to Z^{(s)}$ and $Z^{(s)}\to
\Pi^{(s)}$. The fibers of this map are non-reduced divisors of the
form $pD$ for $D\in \Pi$.

Let us recall that the sheaves $E^i$ are locally free on an open
subset whose intersection with each fiber $Z_t$, $t\in U$,
contains all points in codimension $1$ in $Z_t$ (this follows from
the fact that $E^i_{Z_t}$ are torsion free and $E^i$ is locally
free at the points where $E^i_{Z_t}$ is locally free; see, e.g.,
\cite[Lemma 2.1.7]{HL}).

Let us note that if $X\to Y$  is a faithfully flat morphism of
schemes and for some coherent $\O_Y$-module $G$ its pull-back
$f^*G$ is locally free then $G$ is locally free. This can be
easily seen locally using, e.g., the fact that projective modules
over local rings are free. This also implies that if $f^*G$ is
torsion free then $G$ is also torsion free.

Let us set $U'=F^{s}_{\Pi/k}(U)$. Since $(F^s_{Z/X})^*(E')^i=E^i$
this fact implies that $(E')^i$ are locally free in codimension
$1$ on fibers $Z'_t$ of $p'$ for $t\in U'$. Using \cite[Lemma
2.1.4]{HL} one can in fact see that all factors $(E')^i$ are flat
over $U'$. In particular, for $t\in U'$ the fibres
$(E'_{\bullet})_t$ form a filtration of $E'_t=((q')^*E)_{Z'_t}$
with quotients that are pure $\O_{Z'_t}$-modules. These quotients
are also locally free in codimension $1$ (note that this does not
follow from the purity as the fibers $Z_t$ are non-reduced).

Let us take $t\in U$ and let $t'$ be its image in $U'$. Let $D$
denote the image of $Z_t$ under $p$. Then $Z'_{t'}$ is isomorphic
to the non-reduced divisor $p^sD$. Let $S$ and $T$ denote the
sheaves on $X$ corresponding to $(E'_{l-1})_{Z'_{t'}}$ and
$((E')^{l})_{Z'_{t'}}$. Let $G$ be the kernel of the composition
$E\to E_{p^sD}\to T$. Then we have two short exact sequences:
$$0\to G\to E \to T\to 0$$
and
$$0\to E(-p^sD)\to G \to S\to 0.$$
We can use the first sequence to compute $\Delta (G)$:
$$
\Delta (G) =\Delta (E)-\rho (r-\rho) p^{2s} D^2+2 (r
c_1(T)-(r-\rho)p^sDc_1(E)),$$ where $\rho$ is the rank of $S$ on
$nD$ (which is the same as the rank of $S_D$ on $D$). Now we need
the following lemma:

\begin{Lemma}
Let $D\subset X$ be a smooth divisor on a smooth projective
variety $X$ of dimension at least $2$ and let $d\ge 1$ be an
integer. Let $F$ be a coherent $\O_{dD}$-module that is locally
free at each codimension $1$ point of $D$. Then $c_1(F)=n\cdot
c_1(F_D)$ (as Chern classes of sheaves on $X$).
\end{Lemma}

\begin{proof}
Taking hyperplane sections we can easily reduce the assertion to
the surface case. Let us note that  for every $i\le d$ we have the
following short exact sequences of $\O_{dD}$-modules
$$0\to \O_{(i-1)D}(-D)\to \O_{iD}\to \O_D\to 0.$$
Since $F$ is a locally free $\O_{dD}$-module, the corresponding
sequences remain exact after tensoring with $F$. Therefore we have
a filtration $0=F_0\subset F_1\subset \dots \subset F_{d-1}\subset
F_d =F$ of $F$ with sheaves $F_j=F\otimes \O_{jD}(-(d-j)D)$ and
quotients $F_D(-(p-j)D)$. This easily implies the required
equality of Chern classes.
\end{proof}

\medskip
By the above lemma we have
$$
\Delta (G) H^{n-2} =\Delta (E)H^{n-2}-\rho (r-\rho) p^{2s}
D^2H^{n-2}+2 p^s(r c_1(T_D)-(r-\rho)Dc_1(E))H^{n-2}.$$ From the
definition of the filtration $E_{\bullet}$ we know that
$$(r c_1(T_D)-(r-\rho)Dc_1(E))H^{n-2}=\rho \deg (E^{l}_{Z_t})
-(r-\rho) \deg((E_{l-1})_{Z_t})<0.$$ But we already proved that
both $\deg (E^{l}_{Z_t})$ and $\deg((E_{l-1})_{Z_t})$ are
divisible by $a$ so we get
$$
\Delta (G) H^{n-2} \le \Delta (E)H^{n-2}-\rho (r-\rho) p^{2s}
D^2H^{n-2}-2 a p^s .$$ Using semistability of $E$ and $E(-p^sD)$ we
get
$$\mu _{\max} (G)-\mu (G)=\mu _{\max} (G)-\mu (E)+{r-\rho \over r}
p^sDH^{n-1}\le {r-\rho \over r}ap^s H^n $$
and
$$\mu (G)-\mu _{\min} (G)=\mu (E(-p^sD))-
\mu _{\min} (G)+{\rho \over r}p^sDH^{n-1}\le {\rho \over r}ap^sH^n
.$$ Hence, applying Theorem \ref{Bogomolov} to $G$, we obtain
\begin{eqnarray*}
-\frac{r^2(r-1)^2}{4} {L_X}^2&\le & H^n\cdot \Delta (G)H^{n-2} +r^2(\mu _{\max} (G) -\mu (G))(\mu (G)-\mu_{\min}(G))\\
&\le & H^n\cdot \Delta (E)H^{n-2}-\rho (r-\rho) a^2p^{2s}(H^n)^2-
2a p^s H^n \\
&&+r^2 \left({r-\rho \over r}ap^sH^n \right) \left({\rho \over r}ap^sH^n\right).\\
\end{eqnarray*}
Therefore we get
$$2a p^s \le  \Delta (E)H^{n-2}+\frac{r^2(r-1)^2{L_X}^2}{4H^n} ,$$
which together with inequality (\ref{Flenner}) gives a contradiction
to our assumptions on $a$.
\end{proof}

\medskip

\begin{Example}
For the projective plane $L_{\PP ^2}=0$. In particular, if for $E$
in Theorem \ref{rest-ss} we take the tangent bundle, then the
theorem implies that the restriction of $T_{\PP ^2}$ to a general
curve of degree $\ge 2$ is semistable. Obviously, this result is
the best possible as the restriction of $T_{\PP ^2}$ to any line
is not semistable.
\end{Example}

\medskip

It is easy to see that the technique of proof of Step 5 of Theorem
\ref{rest-ss} gives the following restriction theorem:

\begin{Theorem} \label{Bog-rest-p}
Let $E$ be an $H$-semistable torsion free sheaf and let $a$ be an
integer such that $K_a$ is semistable and
$$a>{\frac{1}{2}}\Delta(E)H^{n-2}+\frac{r^2(r-1)^2{L_X}^2}{8H^n}.$$
Then for a general divisor $D\in|aH|$ the restriction $E_D$ is
$H_D$-semistable.
\end{Theorem}

Note that the above theorem is valid in arbitrary characteristic.
In particular, to get the characteristic zero statement we replace
in Step 5 all Frobenius morphisms by identities (so we consider
only reduced divisors) and following the proof we arrive at
inequality $2a\le \Delta (E)H^{n-2}.$

This improves the bound given in \cite[Corollary 5.4]{La1}
although it works only for a general divisor in $|aH|$. This
theorem, together with Flenner's theorem implies Theorem
\ref{rest-0}. As was noted in Introduction, Theorem \ref{rest-0}
follows also from Theorem \ref{rest-ss} by reduction modulo large
characteristic. More precisely, let $R\subset k$ be a finitely
generated $\ZZ$-algebra such that there exists a projective
morphism $f: \X \to \Spec R$ which after base change to $k$ gives
$X$. Taking appropriate localizations we can assume that $f$ is
smooth. Then there exists a non-negative integer $a$ such that
$T_{\X /R}\otimes \O_{\X/R}(a)$ is $f$-globally generated. This
shows that $L_{\max} (\Omega _{\X_t})$ can be bounded from the
above by the same constant for all fibers $\X_t$ of $f$. Hence
$L_{\X _t}$ tends to $0$ when the characteristic of the residue
field $k(t)$ goes to infinity. Now by openness of semistability
Theorem \ref{rest-ss} implies Theorem  \ref{rest-0}.

\section{Restriction theorem for strong semistability}

Below we sketch a proof of a restriction theorem for strong
semistability for a very general member of sufficiently ample
linear system on some varieties. The proof is a modification of
the proof of \cite[Theorem 2.20]{La2}.

\medskip

\begin{Theorem} \label{rest-sss}
Let $X$ be a smooth $n$-dimensional ($n\ge 2$) variety with a very
ample divisor $H$. Let us assume that $\mu _{\max}(\Omega _X)\le
0$. Let $E$ be an $H$-semistable torsion free sheaf of rank $r\ge
2$ on $X$. Let us take an integer $a$ such that
\begin{eqnarray*}
a> {\frac{1}{2}}\max &\left\{  \Delta(E)H^{n-2}  , \Delta
(\Omega_X)H^{n-2}+n(n-1)^2{(K_XH^{n-1})^2\over  H^n}
\right\}.\\
\end{eqnarray*}
and
$$
\frac{{a+n \choose n}- 1}{a} >\max\{\frac{r^2-1}{4},1\}H^n+1.
$$
Then the restriction of $E$ to the generic hypersurface in $|aH|$
is strongly $H$-semistable.
\end{Theorem}

\begin{proof}
By assumption $\mu _{\max}(\Omega _X)\le 0$, so $E$ is strongly
$H$-semistable and $L_X=0$. By Theorem \ref{Bog-rest-p} and our
assumption on $a$ we also know that $E_D$ is semistable for a
general divisor $D\in |aH|$.

As in the proof of Theorem \ref{rest-ss} let $\Pi$ denote the
complete linear system $|aH|$ and let $Z=\{ (D,x)\in \Pi \times X:
x\in D\}$ be the incidence variety with projections $p: Z\to \Pi$
and $q: Z\to X$. Let $\eta \in \Pi$ be the generic point of $\Pi$
(i.e., a non-closed point whose closure is $\Pi$). We need to
prove that $E_{Z_{\eta}}$ is strongly semistable. If it is not
then there exists $m$ such that the relative Harder--Narasimhan
filtration $0\subset E_0\subset E_1\subset \dots \subset
E_l=q^*((F^m)^*E)$ of $(F^m)^*E$ with respect to $p$ is
non-trivial and the quotients of the Harder--Narasimhan filtration
of the restriction of $E_{Z_{\eta}}$ are strongly semistable. Let
us take such minimal $m$.

If the filtration $E_{\bullet}$ descends under the geometric
Frobenius morphism $Z\to Z^{(1)}$ then the descended filtration
destabilizes $(F^{m-1})^*E$ and the quotients of the
Harder--Narasimhan filtration of the restriction of $E_{Z_{\eta}}$
are strongly semistable, which contradicts our choice of $m$. So
the filtration $E_{\bullet}$ does not descend. Now the canonical
connection $\nabla _{\can}$ on $F^*((F^{m-1})^*E)$ induces a
non-zero $\O_X$-homomorphism
$$E_i\to \Omega_Z\otimes  (F^{m})^*E/E_i.$$
This implies that
$$\mu_{\min}((E_i\otimes ((F^{m})^*E/E_i)^*)_{Z_s})\le \max (
\mu_{\max}((\Omega_{Z/X})_{Z_s}),
\mu_{\max}((q^*\Omega_{X})_{Z_s}))$$ for a general fibre $Z_s$ of
$p$. Now we claim that
$$\mu_{\max}((q^*\Omega_{X})_{Z_s})\le 0.$$
This follows from  $\mu _{\max}(\Omega _X)\le 0$ and the equality
$$\mu _{\max}((\Omega_X)_D)=a\mu _{\max}(\Omega _X)$$
for a general divisor $D\in |aH|$. To prove this last equality let
us consider the Harder--Narasimhan filtration $0=G_0\subset
G_1\subset \dots \subset G_m=\Omega_X$ of the cotangent bundle.
Set $G^i=G_i/G_{i-1}$, $n_i=\rk F_i$ and $\nu _i =\mu (F_i)$.
Since $G^i$ are semistable we have $\Delta (G^i)H^{n-2}\ge 0$ by
Theorem \ref{Bogomolov}. By the Hodge index theorem and
\cite[Lemma 1.4]{La1} we have
\begin{eqnarray*}
{\Delta (\Omega_X)H^{n-2}\over n}&=& \sum {\Delta
(G^i)H^{n-2}\over n_i} -{1\over n}\sum_{i<j} n_in_j \left(
{c_1G^i\over n_i}-{c_1G^j\over n_j}\right) ^2H^{n-2} \\
&\ge& \sum {\Delta(G^i)H^{n-2}\over n_i}-
{1\over n H^n} \sum_{i<j}n_in_j(\nu_i-\nu_j)^2 \\
&\ge & {\Delta(G^l)H^{n-2}\over n_l}+ {n\over H^n} (\mu _{\max}
(\Omega _X)-\mu (\Omega _X))(\mu _{\min} (\Omega _X)-\mu (\Omega
_X))
\end{eqnarray*}
for every integer $l$. Therefore
\begin{eqnarray*}
{\Delta(G^l)H^{n-2}\over n_l}&\le& {\Delta (\Omega_X)H^{n-2}\over
n}+ {n\over H^n} (\mu _{\max} (\Omega _X)-\mu (\Omega _X))(\mu
(\Omega _X)-\mu _{\min}(\Omega _X))\\ &\le& {\Delta
(\Omega_X)H^{n-2}\over n}+ {n(n-1)\over H^n}(\mu (\Omega _X))^2
\end{eqnarray*} and hence Theorem \ref{Bog-rest-p} implies that
$(G^l)_D$ is semistable for a general divisor $D\in |aH|$, which
proves our claim.

Now the rest of proof is as before in proof of Theorem
\ref{rest-ss}: by the Ramanan--Ramanathan theorem we have
$$\mu_i-\mu_{i+1}\le \mu_{\max}((\Omega_{Z/X})_{Z_s})$$
and we know that
$$\frac{a}{\max\{\frac{r^2-1}{4},1 \}}\le \mu_i-\mu_{i+1}$$
and
$$ \mu_{\max}((\Omega_{Z/X})_{Z_s})\le \frac{a^2H^n}{{a+n\choose a}-a-1}.$$
This gives a contradiction with our assumptions on $a$.
\end{proof}

\medskip

One can easily generalize Theorem \ref{rest-sss} to general
complete intersections in varieties with $\mu _{\max}(\Omega
_X)\le 0$ but explicit bounds become harder to write down. One can
also see that instead of $\mu _{\max}(\Omega _X)\le 0$ it is
sufficient to assume that $\mu _{\max}(\Omega _X)\le {1\over
\max\{\frac{r^2-1}{4},1\}}$.

Condition $\mu _{\max}(\Omega _X)\le 0$ is satisfied in many cases
when $K_X$ is not ample. For example, it is satisfied for abelian
varieties, varieties of separated flags and for smooth toric
varieties. Since in characteristic zero the tangent bundle of a
Calabi--Yau variety (i.e., with $K_X=0$) is semistable (for all
polarizations), the reduction of such a variety modulo a general
prime also satisfies assumption of our theorem. More precisely,
one can prove that non-uniruled Calabi--Yau varieties satisfy
condition $\mu _{\max}(\Omega _X)\le 0$  (see \cite{La3}).

\medskip

\section*{Acknowledgements}

The author would like to thank Vikram Mehta and Vijaylaxmi Trivedi
for useful conversations. The author was partially supported by a
Polish MNiSW grant (contract number NN201265333).

\end{document}